\documentclass{groups17}
\makeatletter
\newcommand{\imod}[1]{\allowbreak\mkern4mu({\operator@font mod}\,\,#1)}
\makeatother

\usepackage{tikz}
\usetikzlibrary{calc}
\usetikzlibrary{shapes.geometric}
\usetikzlibrary{trees}

\renewcommand{\a}{\alpha}
\renewcommand{\b}{\beta}

 \newcommand{\e}{\epsilon}
 
 \renewcommand{\O}{\Omega}
 
 \renewcommand{\to}{\rightarrow}

 \newcommand{\C}{\mathcal{C}}

\newcommand{\leqs}{\leqslant}
\newcommand{\geqs}{\geqslant}

\newcommand{\la}{\langle}
\newcommand{\ra}{\rangle}

\newtheorem{con}[thm]{Conjecture} 

\begin{document}

\maketitle{SIMPLE GROUPS, GENERATION AND \\ PROBABILISTIC METHODS}{%
  TIMOTHY C. BURNESS$^{\ast}$}{%
  $^{\ast}$School of Mathematics, University of Bristol, Bristol BS8 1TW, UK \newline
   Email: t.burness@bristol.ac.uk}
    
\begin{abstract}
It is well known that every finite simple group can be generated by two elements and this leads to a wide range of problems that have been the focus of intensive research in recent years. In this survey article we discuss some of the extraordinary generation properties of simple groups, focussing on topics such as random generation, $(a,b)$-generation and spread, as well as highlighting the application of probabilistic methods in the proofs of many of the main results.  
We also present some recent work on the minimal generation of maximal and second maximal subgroups of simple groups, which has applications to the study of subgroup growth and the generation of primitive permutation groups.
\end{abstract}

\section{Introduction}\label{s:intro}

In this survey article we will discuss some of the remarkable generation properties of finite simple groups. Our starting point is the fact that every finite simple group can be generated by just two of its elements (this is essentially a theorem of Steinberg \cite{St}, and the proof requires the Classification of Finite Simple Groups). This leads naturally to a wide range of interesting questions concerning the abundance of generating pairs and their distribution across the group, which have been intensively studied in recent years. Our goal in Sections \ref{s:simple} and \ref{s:spread} is to survey some of the main results and open problems.

Another one of our aims is to highlight the central role played by probabilistic methods. 
In some instances, the given result may already be stated in probabilistic terms (for example, it may refer to the probability that two randomly chosen elements in a group form a generating pair). However, we will see that probabilistic techniques have also been used in an essential way to prove entirely deterministic statements. A striking example is given by Guralnick and Kantor's proof of the so-called \emph{$\frac{3}{2}$-generation} property for simple groups, which we will discuss in Section \ref{s:spread}, together with some far-reaching generalisations.

Our understanding of the subgroups of finite simple groups has advanced greatly in recent years. In particular, many results on the generation of simple groups rely on powerful subgroup structure theorems such as the O'Nan--Scott theorem for alternating groups and Aschbacher's theorem for classical groups. In a different direction, it is natural to consider the generation properties of the subgroups themselves, such as maximal and second maximal subgroups that are located at the top of the subgroup lattice. Indeed, we can seek to understand the extent to which some of the remarkable results for simple groups extend to these subgroups (with suitable modifications, if necessary). The study of problems of this nature was recently initiated through joint work with Liebeck and Shalev and we will discuss some of the main results in Section \ref{s:sub}.

\spc

The coverage of this article is based on the content of my one-hour lecture at the \emph{Groups St Andrews} conference, which was hosted by the University of Birmingham in August 2017. It is a pleasure to thank the organisers for inviting me to give this lecture and for planning and delivering  a very interesting and inspiring meeting. There is a vast literature on the generation of simple groups and so I have had to be very selective in choosing the main topics for this article, which is rather biased towards my own tastes and interests. There are many other excellent survey articles on related topics, which an interested reader may wish to consult. For example, Shalev has written several interesting articles on the use of probabilistic methods in finite group theory, which discuss applications to generation problems and much more (see \cite{shalev}, for example). Liebeck's survey article \cite{Lie} on probabilistic group theory provides an excellent account of some of the more recent developments. 

\spc

Let us say a few words on the notation used in this article. In general, our notation is all fairly standard (and new notation will be defined when needed). It might be helpful to point out that we use the notation of \cite{KL} for simple groups. For example, we will write ${\rm L}_{n}(q) = {\rm PSL}_{n}(q)$ and ${\rm U}_{n}(q) = {\rm PSU}_{n}(q)$ for linear and unitary groups, and we use ${\rm P\O}_{n}^{+}(q)$, etc., for simple orthogonal groups (this differs from the notation used in the Atlas \cite{ATLAS}).

\spc

Finally, I would like to thank Scott Harper for helpful comments on an earlier version of this article.

\section{Generation properties of simple groups}\label{s:simple}

Let $G$ be a finite group and let 
\[
d(G) = {\rm min}\{|S|\,:\, G = \la S \ra \}
\]
be the minimal number of generators for $G$. We will say that $G$ is \emph{$n$-generated} if $d(G)$ is at most $n$. In this section we will focus on the generation properties of simple groups, which is an area of research with a long and rich history. Here the most well-known result is the fact that every finite simple group can be generated by two elements.

\begin{thm}\label{t:sd2}
Every finite simple group is $2$-generated.
\end{thm}

The proof relies on the Classification of Finite Simple Groups. First observe that the alternating groups are straightforward. For example, it is an easy exercise to show that 
$$A_n = \left\{\begin{array}{ll}\langle (1,2,3), (1,2, \ldots, n)\rangle & \mbox{$n$ odd} \\
 \langle (1,2,3), (2,3, \ldots, n)\rangle & \mbox{$n$ even}
 \end{array}\right.$$
In \cite{St}, Steinberg presents explicit generating pairs for each simple group of Lie type. For instance, ${\rm L}_{2}(q) = \la xZ,yZ \ra$, where $Z=Z({\rm SL}_{2}(q))$ and
$$x=\left(\begin{array}{ll}
\a & 0 \\
0 & \a^{-1}
\end{array}\right),\;\; y=\left(\begin{array}{ll}
-1 & 1 \\
-1 & 0
\end{array}\right)$$
with $\mathbb{F}_{q}^{\times}=\la \a \ra$. In \cite{AG}, Aschbacher and Guralnick complete the proof of the theorem by showing that every sporadic group is $2$-generated. 

\spc

In view of Theorem \ref{t:sd2}, there are many natural extensions and variations to consider. For example:
\begin{itemize}
\item[1.] \emph{How abundant are generating pairs in a finite simple group?}
\item[2.] \emph{Can we find generating pairs of prescribed orders?}
\item[3.] \emph{Does every non-identity element belong to a generating pair?}
\end{itemize}
As we shall see, problems of this flavour have been the focus of intensive research in recent years, with probabilistic methods playing a central role in the proofs of many of the main results. 

\subsection{Random generation}\label{s:random}

Let $G$ be a finite group, let $k$ be a positive integer and let
$$\mathbb{P}_k(G) = \frac{|\{(x_1, \ldots, x_k) \in G^k \, :\, G = \la x_1, \ldots, x_k \ra\}|}{|G|^k}$$
be the probability that $k$ randomly chosen elements generate $G$.  For a simple group $G$, Theorem \ref{t:sd2} implies that $\mathbb{P}_{2}(G)>0$ and it is natural to consider the asymptotic behaviour of $\mathbb{P}_{2}(G)$ with respect to $|G|$. This is an old problem, which can be traced all the way back to a conjecture of Netto in 1882. In \cite{Netto}, Netto  writes

\begin{quote}
\emph{``If we arbitrarily select two or more substitutions of $n$ elements, it is to be regarded as extremely probable that the group of lowest order which contains these is the symmetric group, or at least the alternating group."}
\end{quote} 

In other words, Netto is predicting that $\mathbb{P}_2(A_n) \to 1$ as $n$ tends to infinity. This conjecture was proved by Dixon \cite{Dixon} in a highly influential paper published in 1969, which relies in part on the pioneering work of Erd\H{o}s and Tur\'{a}n \cite{ET} on statistical properties of symmetric groups. In the same paper, Dixon makes the bold conjecture that \emph{all} finite simple groups are strongly $2$-generated in the sense of Netto.

\begin{con}[Dixon, 1969]\label{c:dixon} 
\emph{Let $(G_n)$ be any sequence of finite simple groups such that $|G_n|$ tends to infinity with $n$. Then $\mathbb{P}_2(G_n) \to 1$ as $n \to \infty$.}
\end{con}

The proof of Dixon's conjecture was eventually completed in the 1990s. In \cite{KanL}, Kantor and Lubotzky establish the result for classical groups and low rank exceptional groups, and the remaining groups of Lie type were handled by Liebeck and Shalev \cite{LieSh}. The proof is based on the following elementary observations. 

Let $G$ be a finite group and let $\mathcal{M}$ be the set of maximal subgroups of $G$. Set 
\begin{equation}\label{e:zeta}
\zeta_G(s) = \sum_{H \in \mathcal{M}}|G:H|^{-s}
\end{equation}
for a real number $s>0$. For randomly chosen elements $x,y \in G$, observe that $G \neq \langle x,y \rangle$ if and only if $x,y \in H$ for some $H \in \mathcal{M}$. Since $|G:H|^{-2}$ is the  probability that these random  elements are both contained in $H$, it follows that 
$$1-\mathbb{P}_2(G) \leqs \sum_{H \in \mathcal{M}}{|G:H|^{-2}}= \zeta_G(2).$$
Now assume $G$ is a finite simple group of Lie type. By carefully studying $\mathcal{M}$, using powerful results on the subgroup structure of these groups, such as Aschbacher's theorem \cite{asch} for classical groups, one can show that $\zeta_G(2) \to 0$ as $|G|$ tends to infinity. Therefore $\mathbb{P}_2(G) \to 1$ and Dixon's conjecture follows.

It is interesting to note that this probabilistic argument shows that every sufficiently large finite simple group of Lie type is $2$-generated, without the need to explicitly construct a pair of generators. Numerous extensions have since been established. For example, the following striking result is \cite[Theorem 1.1]{MQR}.  

\begin{thm}
We have $\mathbb{P}_2(G) \geqs 53/90$ for every non-abelian finite simple group $G$, with equality if and only if $G=A_6$.
\end{thm}

It turns out that convergence in Conjecture \ref{c:dixon} is rather rapid and strong bounds on $\mathbb{P}_2(G)$ have been established for all simple groups $G$. For example, \cite[Theorem 1.1]{MRD} gives  
$$1 - \frac{1}{n} - \frac{8.8}{n^2} \leqs \mathbb{P}_2(A_n) \leqs 1-\frac{1}{n}-\frac{0.93}{n^2}$$
for all $n \geqs 5$. More generally, by \cite[Theorem 1.6]{LSh3}, there are absolute constants $c_1,c_2>0$ such that 
$$1 - \frac{c_1}{m(G)} \leqs \mathbb{P}_2(G) \leqs 1 - \frac{c_2}{m(G)}$$
for all finite simple groups $G$, where $m(G)$ denotes the minimal index of a proper subgroup of $G$. Note that $m(A_n) = n$. We refer the reader to \cite[Table 4]{GMPS} for a convenient list of the precise values of $m(G)$ when $G$ is a simple group of Lie type.

\subsection{$(a,b)$-generation}\label{s:ab}

Another interesting refinement of Theorem \ref{t:sd2} is to ask if it is possible to find a pair of generators of prescribed orders. With this in mind, for positive integers $a$ and $b$, let us say that a finite group $G$ is \emph{$(a,b)$-generated} if $G = \la x,y \ra$ with $|x|=a$ and $|y|=b$. It is natural to assume that both $a$ and $b$ are primes, at least one of which is odd (since any group generated by two involutions is dihedral). Here the special  case $(a,b)=(2,3)$ is particularly interesting because a group is $(2,3)$-generated if and only if it is a quotient of the modular group ${\rm PSL}_{2}(\mathbb{Z}) \cong Z_2 \star Z_3$. The $(2,3)$-generation problem for simple groups has been widely studied for more than a century and one of the main results is the following.

\begin{thm}
All sufficiently large non-abelian finite simple groups are $(2,3)$-generated, with the exception of ${\rm PSp}_{4}(2^f)$,  ${\rm PSp}_{4}(3^f)$ and ${}^2B_2(q)$.
\end{thm}

This follows from an old theorem of Miller \cite{Mil} for alternating groups, which reveals that $A_n$ is $(2,3)$-generated unless $n \in \{6,7,8\}$. The result for classical groups was proved by Liebeck and Shalev \cite{LSh} and the exceptional groups were handled by L\"{u}beck and Malle in \cite{LM}. More precisely, the latter paper shows that \emph{every} simple exceptional group of Lie type is $(2,3)$-generated, except for $G_2(2)' \cong {\rm U}_{3}(3)$ and of course the Suzuki groups ${}^2B_2(q)$, which do not contain elements of order $3$ (Suzuki \cite{Suz} showed that these groups are $(2,5)$-generated). For completeness, let us record that every sporadic simple group is $(2,3)$-generated, except for ${\rm M}_{11}$, ${\rm M}_{22}$, ${\rm M}_{23}$ and ${\rm McL}$ (see \cite{Wol}). 

For classical groups, Liebeck and Shalev adopt a probabilistic approach in \cite{LSh}, which uses several results on the maximal subgroups of classical groups, such as Aschbacher's theorem \cite{asch} and its extensions. As one might expect, detailed information on the conjugacy classes of elements of order $2$ and $3$ also plays an important role in the proofs. In \cite{LM}, L\"{u}beck and Malle adopt rather different techniques to study the $(2,3)$-generation of exceptional groups. Indeed, their main methods are character-theoretic, using the Deligne--Lusztig theory for reductive groups over finite fields. 

Let us briefly sketch the main ideas in \cite{LSh}. For positive integers $a$ and $b$, let $\mathbb{P}_{a,b}(G)$ be the probability that $G$ is generated by randomly chosen elements of order $a$ and $b$, so $G$ is $(a,b)$-generated if and only if $\mathbb{P}_{a,b}(G)>0$. It is easy to see that 
\begin{equation}\label{e:pab}
1 - \mathbb{P}_{a,b}(G) \leqs \sum_{H \in \mathcal{M}}\frac{i_a(H)i_b(H)}{i_a(G)i_b(G)},
\end{equation}
where $\mathcal{M}$ is the set of maximal subgroups of $G$ as before, and $i_m(X)$ is the number of elements of order $m$ in $X$. By carefully estimating $i_2(H)$ and $i_3(H)$ for $H \in \mathcal{M}$, Liebeck and Shalev show that there is an absolute constant $c$ such that
$$\sum_{H \in \mathcal{M}}\frac{i_2(H)i_3(H)}{i_2(G)i_3(G)} < \sum_{H \in \mathcal{M}}c|G:H|^{-66/65} = c\cdot \zeta_G(66/65)$$
for any finite simple classical group $G \ne {\rm PSp}_{4}(q)$, where $\zeta_G(s)$ is the zeta function in \eqref{e:zeta} (see \cite[Theorems 2.2 and 2.3]{LSh}). The result now follows from \cite[Theorem 2.1]{LSh}, which states that for any $s>1$, $\zeta_G(s) \to 0$ as $|G| \to \infty$ (note that $\zeta_G(1)$ is equal to the number of conjugacy classes of maximal subgroups of $G$, which tends to infinity with $|G|$). Moreover, by combining this result with \cite[Proposition 6.3]{LSh}, we get the following 
natural analogue of Conjecture \ref{c:dixon} for the $(2,3)$-generation of classical groups.

\begin{thm}
For finite simple classical groups $G$, as $|G| \to \infty$ we have
\[
\mathbb{P}_{2,3}(G) \to \left\{\begin{array}{ll} 
0 & \mbox{if $G = {\rm PSp}_4(p^f)$ with $p=2$ or $3$} \\
\frac{1}{2} & \mbox{if $G = {\rm PSp}_4(p^f)$ with $p\ne 2,3$} \\
1 & \mbox{otherwise.}
\end{array}
\right.
\]
\end{thm} 

Using a similar approach, the main theorem of \cite{LSh02} shows that if $a$ and $b$ are primes, not both equal to $2$, then $\mathbb{P}_{a,b}(G) \to 1$ as $|G| \to \infty$, for all simple classical groups $G$ of sufficiently large rank (a sufficient bound on the rank can be given as a function of $a$ and $b$). 

\begin{re}
A complete classification of the $(2,3)$-generated finite simple groups remains out of reach, but there has been significant progress by Di Martino, Pellegrini, Tamburini, Vavilov and others, using constructive methods. For example, Pellegrini \cite{Pel} has very recently resolved the $(2,3)$-generation problem for the linear groups ${\rm L}_{n}(q)$; the only exceptions arise when $(n,q) \in \{(2,9),(3,4),(4,2)\}$, all of which are $(2,5)$-generated. In their recent survey article \cite{PT15}, Pellegrini and Tamburini make the interesting observation that $\O_{8}^{+}(2)$ and ${\rm P\O}_{8}^{+}(3)$ are the only known simple classical groups with natural module of dimension $n \geqs 8$ that are not $(2,3)$-generated. 
\end{re}

To conclude this section, let us briefly discuss the more general $(2,r)$-generation problem. By a theorem of Malle, Saxl and Weigel \cite[Theorem B]{MSW}, every finite simple group $G$ is $(2,r)$-generated for some integer $r \geqs 3$. In fact, for $G \ne {\rm U}_{3}(3)$, it is proved that $G$ is generated by an involution and a strongly real element (that is, an element $x$ so that $x^{-1} = y^{-1}xy$ for some involution $y$), which immediately implies that $G$ is generated by $3$ involutions (one can show that $4$ involutions are needed for ${\rm  U}_{3}(3)$). The following refinement of King \cite{King} shows that $r$ can be taken to be a prime.  

\begin{thm}\label{t:king}
Let $G$ be a non-abelian finite simple group. Then there exists a prime $r$ such that $G$ is $(2,r)$-generated.
\end{thm}

Once again, the proof uses probabilistic methods and we give a brief sketch of the main steps. In view of earlier work, we may  assume $G$ is a classical group over $\mathbb{F}_{q}$, with natural module of dimension $n$. By applying the bound in \eqref{e:pab} (with $a=2$ and $b=5$), King shows that the symplectic groups ${\rm PSp}_{4}(2^f)$ and ${\rm PSp}_{4}(3^f)$ are $(2,5)$-generated. By combining this observation with previous results in the literature, the problem can be reduced to classical groups with $n \geqs 8$. To tackle these groups, we need to recall the notion of a primitive prime divisor.

\begin{de}\label{d:ppd}
For integers $q,e \geqs 2$, a prime divisor $r$ of $q^e-1$ is a \emph{primitive prime divisor} (ppd for short) if $r$ does not divide $q^i-1$ for each $1 \leqs i<e$. 
\end{de}

By a classical theorem of Zsigmondy \cite{Zsig}, such a prime $r$ exists unless $(q,e) = (2^a-1,2)$ or $(2,6)$. Let $r$ be a ppd of $q^e-1$, where $e = e(n)$ is maximal with respect to the condition that $r$ divides $|G|$. For instance, $e = n$ if $G = {\rm L}_{n}(q)$ or ${\rm PSp}_{n}(q)$, and $e=n-2$ if $G = {\rm P\O}_{n}^{+}(q)$. Let $x \in G$ be an element of order $r$ and let $\mathcal{M}(x)$ be the set of maximal subgroups of $G$ containing $x$. 

For $(2,r)$-generation, it suffices to show that $\mathbb{P}_2(G,x)>0$, where $\mathbb{P}_2(G,x)$ is the probability that $x$ and a randomly chosen involution generate $G$. Now 
\begin{equation}\label{e:up}
1-\mathbb{P}_2(G,x) \leqs \sum_{H \in \mathcal{M}(x)} \frac{i_2(H)}{i_2(G)}
\end{equation}
and this essentially reduces the argument to determining $\mathcal{M}(x)$ and then counting the involutions in each $H \in \mathcal{M}(x)$. By choosing $r$ to be a ppd of $q^e-1$ with  $e>n/2$, we can appeal to \cite{GPPS}, which uses Aschbacher's theorem \cite{asch} to determine the maximal subgroups of $G$ containing elements of order $r$. It follows that the subgroups in $\mathcal{M}(x)$ are very restricted and this makes it easier to estimate the upper bound in \eqref{e:up}. This approach is effective in almost all cases, with only a handful of low-dimensional groups requiring further attention (see \cite[Section 7]{King}).

\spc

Notice that King's proof does not yield an absolute bound on the prime $r$ in the statement of the theorem (indeed, $r$ tends to infinity with the rank of the group). However, it is natural to ask if there is an absolute constant $R$ such that every finite simple group is $(2,r)$-generated for some prime $r \leqs R$. In view of the above results, it is not difficult to show that $r \leqs 5$ if $G$ is an alternating or sporadic group, and King's proof shows that $r \leqs 7$ if $G$ is a classical group with natural module of dimension $n \leqs 7$ (the group ${\rm U}_{3}(3)$ is neither $(2,3)$ nor $(2,5)$-generated). The bound $r \leqs 7$ also holds for exceptional groups of Lie type. This leads us naturally to the following conjecture of Conder, which is still open.

\begin{con}[Conder, 2015]
\emph{Every non-abelian finite simple group is $(2,r)$-generated for some $r \in \{3,5\}$, except for ${\rm U}_{3}(3)$, which is $(2,7)$-generated.}
\end{con}

\subsection{Triangle generation}\label{s:tri}

Let $a$, $b$ and $c$ be positive integers with $a \leqs b \leqs c$. We say that a group $G$ is \emph{$(a,b,c)$-generated} if $G = \la x, y \ra$ for elements $x,y \in G$ such that $|x|$, $|y|$ and $|xy|$ divide $a,b$ and $c$, respectively. This is equivalent to the condition that $G$ is a quotient of the \emph{triangle group}
\[
T_{a,b,c} = \la x,y,z \mid x^a = y^b = z^c = xyz = 1 \ra.
\]

The problem of determining the finite simple quotients of triangle groups has attracted significant attention for more than a century. One of the main motivations stems from a famous theorem of Hurwitz from 1893, which states that $|{\rm Aut}(S)| \leqs 84(g-1)$ for any compact Riemann surface $S$ of genus $g \geqs 2$, with equality if and only if ${\rm Aut}(S)$ is a $(2,3,7)$-group (these groups are called \emph{Hurwitz groups}). There has been substantial progress towards a classification of simple Hurwitz groups, but this remains an open problem (see \cite{Con} for a nice survey of results). One of the highlights is the following theorem of Conder \cite{Conder}, which settles a conjecture of Higman from the 1960s asserting that all but finitely many alternating groups are Hurwitz.

\begin{thm}\label{t:con}
The alternating group $A_n$ is a Hurwitz group for all $n \geqs 168$, and for all but $64$ integers in the range $3 \leqs n \leqs 167$.
\end{thm}

For the remainder of this section, we will discuss some more recent results concerning the triangle generation of finite simple groups and related problems.

\spc

Let $G$ be a simple algebraic group over an algebraically closed field $K$ of characteristic $p>0$.  For a fixed triple $(a,b,c)$ of integers, it is natural to ask if there are any values of $r$ such that the corresponding finite quasisimple group $G(p^r)$ is $(a,b,c)$-generated. Here we may assume that $(a,b,c)$ is \emph{hyperbolic}, which means that 
$$\frac{1}{a}+\frac{1}{b}+\frac{1}{c}<1.$$
Indeed, if this condition is not satisfied, then either $T_{a,b,c}$ is soluble, or $(a,b,c)=(2,3,5)$ and $T_{a,b,c} \cong A_5$. In \cite{Mac}, Macbeath proves that ${\rm L}_{2}(p^r)$ is Hurwitz if and only if $r=1$ and $p \equiv 0 ,\pm 1 \imod{7}$, or $r=3$ and $p \equiv \pm 2, \pm 3 \imod{7}$. This is extended by Marion in \cite[Corollary 1]{Marl2}, which states that if $(a,b,c)$ is a hyperbolic triple of primes and $p$ is any prime number, then ${\rm L}_{2}(p^r)$ is $(a,b,c)$-generated if and only if $p^r$ is the smallest power of $p$ such that ${\rm lcm}(a,b,c)$ divides $|{\rm L}_{2}(p^r)|$ (in particular, $r$ is unique). 

To shed further light on Marion's result for ${\rm L}_{2}(p^r)$, we need some additional terminology. For a positive integer $m$, let $j_m(G)$ be the dimension of the subvariety of $G$ of elements of order dividing $m$. Let us say that a triple $(a,b,c)$ of positive integers is \emph{rigid} for $G$ if 
\begin{equation}\label{e:rigid}
j_a(G) + j_b(G) +j_c(G) = 2\dim G.
\end{equation}
We can now state the following conjecture (see \cite[p.621]{Mar10}).  

\begin{con}[Marion, 2010]\label{c:marion}
\emph{Fix a prime $p$ and let $G$ be a simple algebraic group over an algebraically closed field of characteristic $p>0$. If $(a,b,c)$ is a rigid hyperbolic triple of primes for $G$, then there are only finitely many positive integers $r$ such that $G(p^r)$ is $(a,b,c)$-generated.} 
\end{con}

In the special case $G = {\rm PSL}_{2}(K)$ we have $\dim G = 3$ and $j_m(G) = 2$ for all $m \geqs 2$, so every triple is rigid for $G$ and thus the conclusion of the conjecture is in agreement with \cite[Corollary 1]{Marl2}, as stated above. Significant progress towards a proof of the conjecture is made in \cite{Mar10}, where Marion reduces the problem to a handful of cases with 
$$\mbox{$G = {\rm Sp}_{2m}(K)$ (for $m \leqs 13$), ${\rm PSp}_{4}(K)$ or $G_2(K)$.}$$ 

To do this, one first determines the rigid hyperbolic triples of primes for $G$, which is a relatively straightforward exercise using the known dimensions of conjugacy classes of elements of prime order in simple algebraic groups. The rigidity condition in \eqref{e:rigid} is highly restrictive. For instance, if $G$ is an exceptional group, then $G=G_2(K)$ with $(a,b,c) = (2,5,5)$ is the only possibility. 

To complete the reduction, the main step is to eliminate a handful of linear groups $G = {\rm SL}_{n}(K)$. Here the main tool is a well-known theorem of Scott \cite{Scott}, which is used to show that ${\rm GL}_{n}(K)$ has only finitely many orbits on the set
$$\{(x,y,z) \in G^3 \,:\, x^a=y^b=z^c=1,\, xyz=1,\, \mbox{$\la x,y \ra$ irreducible}\},$$
where a subgroup of $G$ is said to be irreducible if it acts irreducibly on the natural module for $G$. Up to conjugacy in ${\rm GL}_{n}(K)$, it follows that there are only finitely many irreducible $(a,b,c)$-generated subgroups of $G$, which immediately gives the desired result in these cases.

Using a completely different approach, Larsen, Lubotzky and Marion \cite{LLM} apply tools from deformation theory to study a generalised version of Conjecture \ref{c:marion}, where the prime condition on the triple $(a,b,c)$ is dropped. The main result is \cite[Theorem 1.7]{LLM}, which establishes this extended form of the conjecture unless $p$ divides $abcd$, where $d$ is the determinant of the Cartan matrix of $G$. This result has recently been pushed further in \cite{JLM}, which proves the extended conjecture for all simple groups $G(p^r)$. As in \cite{Mar10}, the approach in \cite{JLM} relies heavily on the classification of hyperbolic rigid triples (with no prime conditions) and most of the work involves the case $G = {\rm PSp}_{4}(K)$ with $(a,b) = (3,3)$, which requires special attention and different techniques. 

Conjecture \ref{c:marion} for quasisimple groups is still open, even for prime triples. It is also worth noting that the converse to the conjecture is false. For example, $(2,3,7)$ is non-rigid for $G = {\rm SL}_{7}(K)$, but ${\rm SL}_{7}(q)$ is never a Hurwitz group for any prime power $q$ (see \cite[p.623]{Mar10} for further details and examples).

\spc

Natural extensions of triangle generation can be studied by observing that every hyperbolic triangle group $T_{a,b,c}$ is a special type of \emph{Fuchsian group}. This broader family of groups arises naturally in geometry and combinatorial group theory (formally, a Fuchsian group is a finitely generated non-elementary discrete group of isometries of the hyperbolic plane $\mathbb{H}^2$). An orientation-preserving Fuchsian group $\Gamma$ has a rather simple presentation, with generators
$$a_1, b_1, \ldots, a_g, b_g,\; x_1, \ldots, x_d, \; y_1, \ldots, y_s,\; z_1, \ldots, z_t$$
and relations
$$x_1^{m_1} = \cdots = x_d^{m_d} = 1, \; x_1 \cdots x_dy_1 \cdots y_sz_1\cdots z_t[a_1,b_1]\cdots [a_g,b_g] = 1$$
with $g,d,s,t \geqs 0$ and $m_i \geqs 2$ for all $i$. Here $g \geqs 0$ is called the \emph{genus} of $\Gamma$ (the non-orientation-preserving groups admit a similar presentation). In this setting, a triangle group corresponds to the situation where $g=s=t=0$ and $d=3$, so it is natural to extend the notion of triangle generation by considering the finite quotients of arbitrary Fuchsian groups.

A number of remarkable results in this direction have been established in recent years. For instance, Everitt \cite{Everitt} has shown that if $\Gamma$ is an oriented Fuchsian group, then all but finitely many alternating groups are quotients of $\Gamma$. This establishes a  conjecture of Higman (in the oriented case), which generalises Conder's theorem on Hurwitz groups (see Theorem \ref{t:con}). Everitt's approach in \cite{Everitt} builds on the coset-diagram methodology developed by Higman and Conder. By applying very different methods, using a combination of character-theoretic and probabilistic tools, Liebeck and Shalev prove the following theorem, which settles Higman's conjecture in full generality (see \cite[Theorem 1.7]{LSh2004}). In the statement, we use the notation
$${\rm Hom}_{{\rm trans}}(\Gamma, A_n) = \{\varphi \in {\rm Hom}(\Gamma,A_n) \,:\, \mbox{$\varphi(\Gamma)$ is transitive}\}.$$

\begin{thm}
Let $\Gamma$ be a Fuchsian group. Then the probability that a random homomorphism in ${\rm Hom}_{{\rm trans}}(\Gamma, A_n)$ is an epimorphism tends to $1$ as $n \to \infty$.  
\end{thm}

Numerous extensions are pursued in \cite{LSh05}, where $A_n$ is replaced by  a different simple group. For example, the following striking result is \cite[Theorem 1.6]{LSh05}.

\begin{thm}
Let $\Gamma$ be a Fuchsian group of genus $g \geqs 2$ ($g \geqs 3$ if non-oriented), and let $G$ be a finite simple group. Then the probability that a randomly chosen homomorphism in ${\rm Hom}(\Gamma, G)$ is an epimorphism tends to $1$ as $|G| \to \infty$.
\end{thm}

The condition on the genus here is essential, since there are Fuchsian groups of genus $0$ or $1$ which do not have all large enough finite simple groups as quotients. Indeed, we have already seen that if $(a,b,c)$ is a hyperbolic triple of primes then ${\rm L}_{2}(p^r)$ is a quotient of $T_{a,b,c}$ for just one value of $r$. For arbitrary Fuchsian groups, the following conjecture is still open (see \cite[p.323]{LSh05}).

\begin{con}[Liebeck \& Shalev, 2005]
\emph{For any Fuchsian group $\Gamma$ there is an integer $f(\Gamma)$, such that if $G$ is a finite simple classical group of rank at least $f(\Gamma)$, then the probability that a randomly chosen homomorphism in ${\rm Hom}(\Gamma, G)$ is an epimorphism tends to $1$ as $|G| \to \infty$.}
\end{con}

\section{Spread}\label{s:spread}

In the previous section, we highlighted several strong $2$-generation properties of simple groups, which can be viewed as far-reaching generalisations of Theorem \ref{t:sd2}. In this section we study the notions of spread and uniform spread, which provide yet another way to demonstrate the effortless $2$-generation of simple groups. 

\subsection{Definitions}

We begin with the following definition, which was introduced by Brenner and Wiegold \cite{BW} in the 1970s. 

\begin{de}
Let $G$ be a finite group and let $k$ be a positive integer. Then $G$ has \emph{spread} $k$ if for any non-identity elements $x_1, \ldots, x_k \in G$ there exists $y \in G$ such that $G = \la x_i,y\ra$ for all $i$. We say that $G$ is \emph{$\frac{3}{2}$-generated} if it has spread $1$.
\end{de}

One of the main motivations stems from earlier work of Binder \cite{Bin2}, who proved that $S_n$ has spread $2$ for all $n \geqs 5$ (in fact, there is an even earlier result of Piccard \cite{Pic} from the 1930s, which states that both $A_n$ and $S_n$ are $\frac{3}{2}$-generated if $n \geqs 5$). 

Notice that every cyclic group has spread $k$ for all $k \in \mathbb{N}$, so for the remainder of Section \ref{s:spread} we will  assume $G$ is non-cyclic. Set
$$s(G) = \max\{k \in \mathbb{N}_{0} \, : \, \mbox{$G$ has spread $k$}\}.$$
In practice, it can often be more convenient to work with the more restrictive notion of \emph{uniform spread}, which was formally introduced much more recently in \cite{BGK}.

\begin{de}
A finite group $G$ has \emph{uniform spread} $k$ if there exists a fixed conjugacy class $C$ of $G$ such that for any non-identity elements $x_1, \ldots, x_k \in G$ there exists $y \in C$ such that $G = \la x_i,y\ra$ for all $i$.
\end{de}

For a non-cyclic group $G$, set 
$$u(G) = \max\{k \in \mathbb{N}_{0} \, : \, \mbox{$G$ has uniform spread $k$}\}$$
and observe that $u(G) \leqs s(G) < |G|-1$. Note that the first inequality can be strict; for 
example, $u(S_6)=0$ and $s(S_6)=2$. 

\subsection{The spread of simple groups}\label{ss:spreadd}

In \cite{BW}, Brenner and Wiegold extend the earlier work of Binder and Piccard by investigating the spread of various families of simple groups. Among several interesting results, they prove that $s(A_{2n})=4$ for all $n \geqs 4$ and they show that 
$$s({\rm L}_{2}(q)) = \left\{\begin{array}{ll}
q-1 & \mbox{if $q \equiv 1 \imod{4}$} \\
q-4 & \mbox{if $q \equiv 3 \imod{4}$} \\
q-2 & \mbox{if $q$ is even} 
\end{array}\right.$$
for $q \geqs 11$ (see \cite[Theorems 3.10 and 4.02]{BW}). In particular, the spread of a finite simple group can be arbitrarily large. 

They also observe that the spread of odd degree alternating groups is radically different. For instance, \cite[Theorem 4.01]{BW} states that 
\begin{equation}\label{e:a19}
6098892799 \leqs s(A_{19}) \leqs  6098892803.
\end{equation}
Later work by Guralnick and Shalev \cite{GSh} shows that $s(A_{p})$ tends to infinity with $p$ when $p$ is a prime number. More generally, \cite[Theorem 1.1]{GSh} implies that if $(G_i)$ is a sequence of alternating groups such that $G_i = A_{n_i}$ and $n_i$ tends to infinity with $i$, then $s(G_i) \to \infty$ if and only if $f(n_i) \to \infty$, where $f(n_i)$ is the smallest prime divisor of $n_i$. 

\spc

In Steinberg's original paper \cite{St}, where he presents a generating pair for each simple group of Lie type, he suggests that these groups may have the much stronger $\frac{3}{2}$-generation property (he is aware of Piccard's result for alternating groups). Steinberg's prediction was eventually verified almost 40 years later (in a stronger form) by Stein \cite{Stein}, and independently by Guralnick and Kantor \cite{GK}.

\begin{thm}\label{t:ug1}
If $G$ is a non-abelian finite simple group, then $u(G) \geqs 1$.
\end{thm}

Stronger results are established in \cite{GK}, which turn out to be important for subsequent improvements of the bound in Theorem \ref{t:ug1}. More precisely, it is shown that there is a conjugacy class $C$ of $G$ such that each non-identity element of $G$ generates $G$ with at least $1/10$ of the elements in $C$, and they also establish some related results for almost simple groups. In later work \cite{BGK}, Breuer, Guralnick and Kantor show that the constant $1/10$ can be improved to $13/42$ (for $G = \O_8^{+}(2)$, this is best possible), and by excluding a short list of known groups, $1/10$ can be replaced by $2/3$. 

As we shall see below, the fact that the above fraction $1/10$ can be replaced by $2/3$ in almost all cases is the key ingredient in the proof of the following theorem, which is the main result on the spread of simple groups (see \cite[Corollary 1.3]{BGK}).

\begin{thm}\label{t:bgk}
Let $G$ be a non-abelian finite simple group. Then $u(G) \geqs 2$, with equality if and only if
$G \in \{A_5, A_6, \O_{8}^{+}(2), {\rm Sp}_{2m}(2) \, (m \geqs 3)\}$.
\end{thm}

The proof of Theorem \ref{t:bgk} uses probabilistic methods, based on fixed point ratio estimates. To describe the main ideas, we need some notation.

Let $G$ be a finite group. For $x,y \in G$, let 
$$\mathbb{P}(x,y)= \frac{|\{z \in y^G \, : \, G=\langle x,z \rangle \}|}{|y^G|}$$
be the probability that $x$ and a randomly chosen conjugate of $y$ generate $G$. Set 
$$Q(x,y) = 1- \mathbb{P}(x,y).$$
For a subgroup $H \leqs G$ and element $x \in G$, let 
$${\rm fpr}(x,G/H) = \frac{|x^G \cap H|}{|x^G|}$$
be the \emph{fixed point ratio} of $x$. This is the proportion of fixed points of $x$ under the natural action of $G$ on the set of cosets of $H$ in $G$. Notice that ${\rm fpr}(x,G/H) \leqs {\rm fpr}(x^m,G/H)$ for all $m \in \mathbb{N}$.

\begin{lemma}\label{lp1}
Suppose there exists an element $y \in G$ and a positive integer $k$ such that 
$Q(x,y)<1/k$ for all $1 \ne x \in G$. Then $u(G) \geqs k$.
\end{lemma}

\begin{proof}
Let $x_1, \ldots, x_k \in G$ be non-identity elements and set $E = E_1 \cap \cdots \cap E_k$, where $E_i$ is the event that $G=\langle x_i,z \rangle$ for a randomly chosen conjugate $z \in y^G$. Then
\begin{align*}
 \mathbb{P}(E) = 1-\mathbb{P}(\bar{E}) & =  1- \mathbb{P}(\bar{E}_1 \cup \cdots \cup \bar{E}_k) \\
 & \geqs 1- \sum_{i=1}^k \mathbb{P}(\bar{E}_i) = 1- \sum_{i=1}^k Q(x_i,y) >1 - k \cdot \frac{1}{k} = 0
 \end{align*}
and the result follows.
\end{proof}

Notice that if we can find a conjugacy class $C = y^G$ with the property that each non-identity element of $G$ generates $G$ with at least $2/3$ of the elements in $C$, then $Q(x,y)<1/3$ for all $1 \ne x \in G$ and thus $u(G) \geqs 3$ by Lemma \ref{lp1}. This is the main strategy adopted in the proof of Theorem \ref{t:bgk}.  

In order to effectively apply Lemma \ref{lp1}, we need to be able to estimate the probability $Q(x,y)$. Here the key result is the following lemma (as before, we define $\mathcal{M}(y)$ to be the set of maximal subgroups of $G$ containing $y$).  

\begin{lemma}\label{c:mgen}
For $x,y \in G$, we have  
$$Q(x,y) \leqs \sum_{H \in \mathcal{M}(y)}{\rm fpr}(x,G/H).$$
\end{lemma}

\begin{proof}
If $z \in y^G$ then $G \neq \la x,z \ra$ if and only if $\la x',y \ra \leqs H$ for some $x' \in x^G$ and $H \in \mathcal{M}(y)$. Therefore, 
$$Q(x,y) \leqs \sum_{H \in \mathcal{M}(y)}\mathbb{P}_{x}(H),$$
where
$$\mathbb{P}_{x}(H)= \frac{|x^G \cap H|}{|x^G|} = {\rm fpr}(x,G/H)$$
is the probability that a random conjugate of $x$ lies in $H$. The result follows.
\end{proof}

If we can find an element $y \in G$ and a positive integer $k$ such that 
$$\sum_{H \in \mathcal{M}(y)}{\rm fpr}(x,G/H) < \frac{1}{k}$$
for all $x \in G$ of prime order, then by combining Lemmas \ref{lp1} and \ref{c:mgen} we deduce that $u(G) \geqs k$. To do this effectively, we need to identify an element $y \in G$ that is contained in very few maximal subgroups of $G$, and we need to be able to 
determine the subgroups in $\mathcal{M}(y)$. We then require upper bounds on the appropriate fixed point ratios for elements of prime order. Such bounds are useful in many different contexts and there is an extensive literature to draw upon. For example, if $G$ is a group of Lie type over $\mathbb{F}_q$ then there is the general upper bound ${\rm fpr}(x,G/H) \leqs 4/3q$ due to Liebeck and Saxl \cite{LSax} (with prescribed exceptions). See \cite{Bur1, LLS} and 
\cite[Section 3]{GK} for stronger bounds in special cases.

Notice that there is some considerable flexibility in this approach. There is not always an obvious candidate for $y$, and in practice there may be many valid possibilities (although some choices will require more work than others in estimating the upper bound in Lemma \ref{c:mgen}). 

To illustrate some of the main ideas in the proof of Theorem \ref{t:bgk}, let us look at three examples.

\begin{ex}
Suppose $G = A_n$, where $n \geqs 8$ is even. We will use Lemma \ref{c:mgen} to show that $u(G) \geqs 3$ (recall that $s(G)=4$ by \cite[Theorem 3.10]{BW}). Set $n=2m$, $k=m-(2,m-1)$ and 
$$y = (1,2, \ldots, k)(k+1, \ldots, n) \in G.$$
First we determine the maximal overgroups in $\mathcal{M}(y)$. To do this, suppose 
$H \in \mathcal{M}(y)$ and consider the action of $H$ on $\{1, \ldots, n\}$. If $H$ is intransitive, then it is clear that $H = (S_k \cap S_{n-k}) \cap G$ is the only possibility (that is, $H$ has to be the setwise stabiliser in $G$ of $\{1, \ldots, k\}$). Since $k$ and $n-k$ are coprime, it is easy to rule out imprimitive subgroups, so we may assume $H$ is primitive. Here it is helpful to observe that $y^{n-k}$ is a $k$-cycle and $1<k<n/2$, so a classical theorem of Marggraf from 1889 (see \cite[Theorem 13.5]{Wie}) implies that $H=G$ and we reach a contradiction. Therefore, $\mathcal{M}(y) = \{H\}$ with $H = (S_k \cap S_{n-k}) \cap G$, and the action of $G$ on $G/H$ is equivalent to the action of $G$ on the $k$-element subsets of $\{1, \ldots, n\}$. It is straightforward to show that ${\rm fpr}(x,G/H) < 1/3$ for all $x \in G$ of prime order (see the proof of \cite[Proposition 6.3]{BGK}), whence
$$\sum_{H \in \mathcal{M}(y)}{\rm fpr}(x,G/H) < \frac{1}{3}$$
and thus $u(G) \geqs 3$ via Lemmas \ref{lp1} and \ref{c:mgen}.
\end{ex}

\begin{re}
The analysis of odd degree alternating groups is more complicated. In this situation, we cannot choose an element $y \in A_n$ with exactly two cycles, so one may be forced to work with an element that is contained in several maximal subgroups. Still, some special cases are easy to handle. For example, if $G = A_{19}$ and $y$ is a $19$-cycle, then $\mathcal{M}(y) = \{H\}$ with 
$$H = N_G(\la y \ra) = {\rm AGL}_{1}(19) \cap G = Z_{19}{:}Z_9$$
and one can check that 
$${\rm fpr}(x,G/H) \leqs \frac{1}{6098892800}$$
for all $x \in G$ of prime order (with equality if $x \in G$ has cycle-shape $[3^6,1]$). Therefore, $u(G) \geqs 6098892799$, which agrees with the lower bound on $s(G)$ in \eqref{e:a19}.
\end{re}

\begin{ex}
Suppose $G=E_8(q)$ and let $y \in G$ be a generator of a maximal torus of order $r=q^8+q^7-q^5-q^4-q^3+q+1$. By a theorem of Weigel (see case (j) in \cite[Section 4]{Wei}), we have $\mathcal{M}(y) = \{H\}$ with $H = N_G(\la y \ra) = Z_{r}{:}Z_{30}$. Since $|x^G|>q^{58}$ for all $x \in G$ of prime order (minimal if $x$ is a long root element), we deduce that 
$${\rm fpr}(x,G/H) = \frac{|x^G \cap H|}{|x^G|} <  \frac{|H|}{q^{58}} < q^{-44}$$
and thus $u(G) \geqs q^{44}$. 
\end{ex}

\begin{ex}
Suppose $G = {\rm PSp}_{2m}(q)$ is a symplectic group, where $m \geqs 6$ is even and $q$ is odd. Let $V$ be the natural module for $G$. Following \cite[Proposition 5.10]{BGK}, fix a semisimple element $y \in G$ that preserves an orthogonal decomposition $V = U \perp W$, where $U$ and $W$ are nondegenerate subspaces of dimension $4$ and $2m-4$, respectively. Moreover, assume that $y$ acts irreducibly on both $U$ and $W$. Since $U$ and $W$ are the only proper nonzero subspaces of $V$ preserved by $y$, it follows that the stabiliser of $U$ in $G$ (a subgroup of type ${\rm Sp}_{4}(q) \times {\rm Sp}_{2m-4}(q)$) is the only reducible subgroup in $\mathcal{M}(y)$. 

In order to determine the irreducible subgroups in $\mathcal{M}(y)$, it is very helpful to observe that $|y|$ is divisible by a primitive prime divisor of $q^{2m-4}-1$ (see Definition \ref{d:ppd}). Recall that the subgroups of classical groups containing such elements are studied in \cite{GPPS}, where the analysis is organised according to Aschbacher's subgroup structure theorem (see \cite{asch}). By carefully applying the main theorem of \cite{GPPS}, it is possible to severely restrict the subgroups in $\mathcal{M}(y)$. Indeed, one can show that there are only two irreducible subgroups in $\mathcal{M}(y)$, both of which are field extension subgroups of type ${\rm Sp}_{m}(q^2)$ (see \cite[Proposition 5.10]{BGK}). We now need to estimate fixed point ratios for the appropriate actions. 

Suppose $x \in G$ has prime order. If $H$ is the reducible subgroup of type ${\rm Sp}_{4}(q) \times {\rm Sp}_{2m-4}(q)$, then \cite[Proposition 3.16]{GK} gives
$${\rm fpr}(x,G/H) < 2q^{2-m} + q^{-m}+q^{-2}+q^{4-2m}.$$
Similarly, if $H$ is of type ${\rm Sp}_{m}(q^2)$ then \cite[Lemma 3.4]{BGK} yields
$${\rm fpr}(x,G/H) < q^{3-2m}.$$
Putting all this together, we conclude that 
$$\sum_{H \in \mathcal{M}(y)}{\rm fpr}(x,G/H)< 2q^{2-m} + q^{-m}+q^{-2}+q^{4-2m}+2q^{3-2m} < \frac{1}{3}$$
for all $m \geqs 6$ and $q \geqs 3$, whence $u(G) \geqs 3$.
\end{ex}

The spread of sporadic groups has also been the subject of several papers, giving upper and lower bounds (see \cite{BH,BM,F1,F2} for example). For example, the best known result on the spread of the Monster $\mathbb{M}$ gives  
$$3385007637938037777290624 \leqs s(\mathbb{M}) \leqs 5791748068511982636944259374$$
(see \cite[Theorem 1]{F2}). It is interesting to note that ${\rm M}_{11}$ and ${\rm M}_{23}$ are the only sporadic simple groups for which the exact spread has been computed: we have $s({\rm M}_{11})=3$ and $s({\rm M}_{23}) = 8064$. 

\subsection{Almost simple groups and generating graphs}

We can extend the study of spread and uniform spread to the broader class of almost simple groups. Recall that a finite group $G$ is \emph{almost simple} if 
\[ 
T \leqs G \leqs {\rm Aut}(T)
\]
for some non-abelian finite simple group $T$ (the socle of $G$). The following theorem on minimal generation is due to Dalla Volta and Lucchini (see \cite[Theorem 1]{DL}).

\begin{thm}\label{t:dl}
Let $G$ be an almost simple group with socle $T$. Then 
\[
d(G) = \max\{2,d(G/T)\}\leqs 3.
\]
\end{thm}
  
It is not difficult to see that this bound is best possible. For instance, if we take $G = {\rm Aut}({\rm L}_{n}(q))$, where $nq$ is odd and $q=p^{2f}$ with $p$ a prime, then the elementary abelian group $(Z_2)^3$ is a homomorphic image of $G$, whence $d(G)=3$.

Let $G$ be an almost simple group with socle $T$. In this more general setting, it is still possible to establish a slightly weaker spread-two property. Indeed, \cite[Corollary 1.5]{BGK} states that for any pair of non-identity elements $x_1,x_2 \in G$, there exists $y \in G$ such that $\la x_i,y \ra$ contains $T$, for $i=1,2$. Of course, some sort of modified statement is needed because 
$s(G)=0$ if $G/T$ is non-cyclic (indeed, if $1 \ne x \in T$ and $G = \la x,y \ra$ for some $y \in G$, then $G/T = \la Ty \ra$). In view of this observation, it is interesting to consider the spread and uniform spread of almost simple groups of the form $G = \la T, x \ra$ for some automorphism $x$ of $T$. 

Further motivation for studying this situation comes from a remarkable conjecture of Breuer, Guralnick and Kantor (see \cite[Conjecture 1.8]{BGK}).

\begin{con}[Breuer et al., 2008]\label{c:bgk}
\emph{Let $G$ be a finite group. Then $s(G) \geqs 1$ if and only if $G/N$ is cyclic for every nontrivial normal subgroup $N$ of $G$.}
\end{con}

In recent work, Guralnick has established a reduction of this conjecture to almost simple groups and various special cases have been established. For instance, almost simple sporadic groups are handled in \cite{BGK}, while the desired result for symmetric groups was proved by Binder \cite{Bin2} (as previously noted). More precisely, these results show that $s(G) \geqs 2$ for every almost simple group $G$ with an alternating or sporadic socle $T$ and cyclic quotient $G/T$. For groups of Lie type, progress so far has focussed on certain families of classical groups, starting with the main theorem of \cite{BG}, which shows that $s(G) \geqs 2$ when $T = {\rm L}_{n}(q)$. 

To do this, our initial aim is to establish the bound $u(G) \geqs 2$ using the same probabilistic approach as before, via Lemmas \ref{lp1} and \ref{c:mgen}. Although the underlying strategy is the same, the details in the almost simple setting are significantly more complicated. Indeed, if our given group is $G = \la T,x \ra$ then we have to identify a suitable conjugacy class $y^G$ for some element $y$ in the coset $Tx$. Here the main challenge is to determine the maximal overgroups of such an element $y$ and various techniques are needed to do this, which depend on the specific type of automorphism $x$. For instance, if $x$ is a field automorphism, then we use the theory of \emph{Shintani descent} for algebraic groups to identify an appropriate element $y \in Tx$ (see \cite[Section 2.6]{BG} for further details). 

Using similar methods, the results in \cite{BG} have recently been extended by Harper \cite{Harper} to the classical groups with socle $T = {\rm PSp}_{n}(q)$ and $\O_n(q)$ (with $nq$ odd in the latter case). Further work to complete the analysis of almost simple groups of Lie type is in progress, with the ultimate goal of completing the proof of Conjecture \ref{c:bgk}.

\spc

Finally, to conclude this section we briefly explain how some of the above results can be cast in terms of the \emph{generating graph} of a finite group, which leads to some interesting open problems.

\begin{de}
Let $G$ be a finite group. The \emph{generating graph} $\Gamma(G)$ is a graph on the 
non-identity elements of $G$ so that two vertices $x,y$ are joined by an edge if and only if $G=\la x,y \ra$.
\end{de}

For a $2$-generated group $G$, this graph encodes some interesting  generation properties of the group. For example, $G$ is $\frac{3}{2}$-generated if and only if $\Gamma(G)$ has no isolated vertices. Similarly, if $s(G) \geqs 2$ then $\Gamma(G)$ is connected with diameter at most $2$. In this way, we obtain an appealing interplay between groups and graphs, leading to a number of natural questions. For instance, what is the (co)-clique number and chromatic number of $\Gamma(G)$? Does $\Gamma(G)$ contain a Hamiltonian cycle (i.e. a cycle that visits every vertex exactly once)? etc. The following theorem brings together some of the main results on the generating graph of a finite simple group.

\begin{thm}\label{t:gg1}
Let $G$ be a non-abelian finite simple group and let $\Gamma(G)$ be its generating graph.
\begin{itemize}
\item[{\rm (i)}] $\Gamma(G)$ has no isolated vertices.
\item[{\rm (ii)}] $\Gamma(G)$ is connected and has diameter $2$.
\item[{\rm (iii)}] $\Gamma(G)$ contains a Hamiltonian cycle if $|G|$ is sufficiently large.
\end{itemize}
\end{thm}

\begin{proof}
Clearly, (ii) implies (i), and (ii) is an immediate corollary of Theorem \ref{t:bgk}. Part (iii) is \cite[Theorem 1.2]{BGLMN}. 
\end{proof}

It is worth noting that the proof of (iii) uses probabilistic methods in the sense that it relies on the proof of Dixon's conjecture (see Conjecture \ref{c:dixon}). Roughly speaking, if $|G|$ is large then $\mathbb{P}_2(G)$ is close to $1$, which translates into lower bounds on the degrees of the vertices in $\Gamma(G)$. If these bounds are sufficiently large (relative to $|G|$), then one can appeal to \emph{P\'{o}sa's criterion} (in our setting, if $m=|G|-1$ and $d_1 \leqs \cdots \leqs d_m$ are the vertex degrees, then the condition we need is $d_k \geqs k+1$ for all $1 \leqs k < m/2$) to force the existence of a Hamiltonian cycle and this is how the proof of (iii) proceeds in \cite{BGLMN}.

It is conjectured that the generating graph of \emph{every} non-abelian finite simple group contains a Hamiltonian cycle. In fact, the following stronger conjecture is proposed in \cite{BGLMN} (see \cite[Conjecture 1.6]{BGLMN}).

\begin{con}[Breuer et al., 2010]\label{c:bet}
\emph{Let $G$ be a finite group with $|G| \geqs 4$. Then $\Gamma(G)$ contains a Hamiltonian cycle if and only if $G/N$ is cyclic for every nontrivial normal subgroup $N$ of $G$.}
\end{con} 

Notice that the condition on quotients here is identical to the one in Conjecture \ref{c:bgk}. Of course, it is clear that this is a necessary condition for Hamiltonicity, but it is rather striking that it is also conjectured to be sufficient. By \cite[Proposition 1.1]{BGLMN}, the conjecture holds for all soluble groups and it has been verified for the simple groups ${\rm L}_{2}(q)$ (see \cite[Section 6]{Bnotes}). There has also been recent progress for alternating groups by Erdem \cite{Erdem}, who has proved that $\Gamma(A_n)$ is Hamiltonian if $n \geqs 4100$. 

We finish by formulating another conjecture, which combines and strengthens Conjectures \ref{c:bgk} and \ref{c:bet}.

\begin{con}
Let $G$ be a finite group with $|G| \geqs 4$. Then the following are equivalent:
\begin{itemize}
\item[{\rm (i)}] $G$ has spread $1$.
\item[{\rm (ii)}] $G$ has spread $2$.
\item[{\rm (iii)}] $\Gamma(G)$ has no isolated vertices.
\item[{\rm (iv)}] $\Gamma(G)$ is connected. 
\item[{\rm (v)}] $\Gamma(G)$ is connected with diameter at most $2$.
\item[{\rm (vi)}] $\Gamma(G)$ contains a Hamiltonian cycle.
\item[{\rm (vii)}] $G/N$ is cyclic for every nontrivial normal subgroup $N$.
\end{itemize}
\end{con}

Notice that this conjecture implies that there is no finite group with $s(G)=1$. It also gives the following remarkable dichotomy for generating graphs: either $\Gamma(G)$ has an isolated vertex, or it is connected with diameter at most $2$. Given the proof of Conjecture \ref{c:bet} for soluble groups in \cite{BGLMN}, it is not too difficult to verify the conjecture in the soluble case. However, it is very much an open problem for insoluble groups.

\section{Generating subgroups of simple groups}\label{s:sub}

In this final section, we investigate the generation properties of subgroups of simple groups. As we have seen repeatedly in Sections \ref{s:simple} and \ref{s:spread}, many questions concerning the generation of simple groups can be reduced to problems involving their maximal subgroups. However, there are very few results in the literature on the generation properties of these subgroups themselves. For example, it is natural to consider the extent to which some of the familiar results for simple groups (such as $2$-generation and random generation, as in Dixon's conjecture) can be extended to certain subgroups of interest, with appropriate modifications (if necessary). The goal of this section is to address some of these questions; the main references are \cite{BLS1} and \cite{BLS2}.

\spc

Let $G$ be a finite group and recall that $d(G)$ denotes the minimal number of generators for $G$. Notice that $d$ is not a monotonic function, in the sense that a subgroup $H$ of $G$ may require more generators. For instance, if $n$ is even, then the elementary abelian subgroup $\la (1,2), (3,4),\ldots, (n-1,n)\ra$ of $S_n$ needs $n/2$ generators, but  $S_n = \la (1,2),(1,2,\ldots, n) \ra$ is $2$-generated. In this setting, there is an attractive theorem of McIver and Neumann (see \cite[Lemma 5.2]{MN}), which states that 
$\max\{d(H) \,:\, H \leqs S_n\} = \lfloor n/2 \rfloor$
for all $n \geqs 4$, so there is some control on the required number of generators for a subgroup of $S_n$. More generally, we can bound $d(H)$ in terms of $d(G)$ and its index $[G:H]$.

\begin{lemma}\label{l:ns}
If $G$ is a finitely generated group and $H \leqs G$ has finite index, then 
$$d(H) \leqs  [G:H](d(G)-1)+1.$$
\end{lemma}

\begin{proof}
Let $F$ be the free group on $d(G)$ generators, so $G = F/K$ and $H = L/K$ for some subgroups $K \leqs L \leqs F$. Since $[F:L] = [G:H]$ is finite, the Nielsen--Schreier index formula implies that $d(L) = [F:L](d(F)-1)+1$. Therefore
$$d(H) = d(L/K) \leqs d(L) = [G:H](d(G)-1)+1$$
and the result follows.
\end{proof}

The following example shows that the upper bound in Lemma \ref{l:ns} is sharp, even for maximal subgroups.

\begin{ex}\label{ex:max}
Let $p \geqs 3$ be a prime and consider the group $G= (Z_2)^{p+1}{:}Z_p$, where $Z_p$ cyclically permutes the first $p$ copies of $Z_2$ in the direct product $(Z_2)^{p+1}$. Set $H = (Z_2)^{p+1}$. Then $H$ is a maximal subgroup of $G$ and it is easy to see that $d(G)=2$ and $d(H)=p+1=[G:H]+1$.
\end{ex}

\subsection{Maximal subgroups}\label{s:max}

Let $G$ be an almost simple group with socle $T$ and recall that $d(G) \leqs 3$ (see Theorem \ref{t:dl}). We begin our investigation of the  generation properties of subgroups of $G$ by starting at the top of the subgroup lattice with the maximal subgroups. The minimal generation of these subgroups is systematically studied in \cite{BLS1} and the main result is the following theorem (see \cite[Theorem 2]{BLS1}), which reveals that every maximal subgroup of $G$ can also be generated by very few elements.

\begin{thm}\label{t:maxgen}
Let $G$ be an almost simple group with socle $T$ and let $H$ be a maximal subgroup of $G$. Then $d(H \cap T) \leqs 4$ and $d(H) \leqs 6$. In particular, every 
maximal subgroup of a finite simple group is $4$-generated.
\end{thm}

It is not too difficult to show that the bound $d(H \cap T) \leqs 4$ is best possible in the sense that there are infinitely many examples for which equality holds (see Examples \ref{ex:dh5} and  \ref{ex:diag}). However, it is possible that the bound $d(H) \leqs 6$ can be improved. For instance, if $T$ is an alternating group then the proof of Theorem \ref{t:maxgen} already gives $d(H) \leqs 4$ (see Proposition \ref{p:an} below) and similarly $d(H) \leqs 3$ if $T$ is a sporadic group. We can construct examples with $d(H)=5$ when $T$ is a classical group (the author thanks Dr. Gareth Tracey for drawing his attention to the following example).

\begin{ex}\label{ex:dh5}
Suppose $G = \la T, x \ra = T.2$, where $T = {\rm P\O}_{n}^{+}(q)$, $q=q_0^2$ is odd and $x$ is an involutory field automorphism of $T$. In addition assume $n=a^2$, where $a \geqs 6$ and $a \equiv 2 \imod{4}$. Then $G$ has a maximal subgroup $H$ of type $O_a^{+}(q) \wr S_2$ (in the terminology of \cite{KL}, this is a tensor product subgroup in Aschbacher's $\C_7$ collection) with precise structure 
$$H = ({\rm P\O}_{a}^{+}(q) \times {\rm P\O}_{a}^{+}(q)).(Z_2)^5.$$
Clearly, $d(H) \geqs 5$ and one can check that equality holds. Note that $d(H \cap T) = 4$ in this case, which demonstrates the sharpness of the first bound in Theorem \ref{t:maxgen}.
\end{ex}

However, it is not known if there are any examples with $d(H)>5$ when $T$ is a simple group of Lie type.

\begin{con}\label{c:as}
\emph{Every maximal subgroup of an almost simple group is $5$-generated.}
\end{con}

Not surprisingly, subgroup structure theorems for almost simple groups play an essential role in the proof of Theorem \ref{t:maxgen}. As previously noted, the $2$-generation of finite simple groups is established by inspecting the list of groups provided by the Classification Theorem. The situation for maximal subgroups is not quite as clear-cut because, in general, we cannot consult a complete list of subgroups. However, we do have access to some powerful reduction theorems, such as Aschbacher's theorem \cite{asch} for finite classical groups (combined with the detailed structural information in \cite{KL} and the comprehensive treatment of the low-dimensional classical groups in \cite{BHR}) and theorems of Liebeck, Seitz and others for exceptional groups of Lie type (see \cite{LS03}). These results can be viewed as Lie type analogues of the 
O'Nan--Scott theorem (see Theorem \ref{t:ons}), which is the main tool for handling the alternating and symmetric groups. All of these results partition the maximal subgroups of an almost simple group into several families, providing a coherent framework for the proof of Theorem \ref{t:maxgen}. It is important to note that for the purposes of this proof, we do not need to worry about any unknown almost simple maximal subgroups because every almost simple group is $3$-generated by Theorem \ref{t:dl}.

\spc

Let us sketch a proof of Theorem \ref{t:maxgen} in the case where $T$ is an alternating group.
First we recall the O'Nan--Scott theorem, which describes the maximal subgroups of $G$ (see \cite[Theorem 4.1A]{DM}). 

\begin{thm}[O'Nan--Scott]\label{t:ons}
Let $G=A_n$ or $S_n$, and let $H$ be a maximal subgroup of $G$. Then one of the following holds:
\begin{itemize}
\item[{\rm (i)}] $H$ is intransitive: $H=(S_{k} \times S_{n-k})\cap G$, $1 \leqs k<n/2$;
\item[{\rm (ii)}] $H$ is affine: $H={\rm AGL}_{d}(p) \cap G$, $n=p^d$, $p$ prime, $d \geqs 1$; 
\item[{\rm (iii)}] $H$ is imprimitive or wreath-type: $H=(S_k \wr S_t)\cap G$, $n = kt$ or $k^t$, $k,t \geqs 2$;
\item[{\rm (iv)}] $H$ is diagonal: $H=(A^k.({\rm Out}(A) \times S_k))\cap G$, $A$ non-abelian simple, $n = |A|^{k-1}$;
\item[{\rm (v)}] $H$ is almost simple.
\end{itemize}
\end{thm}

\begin{prop}\label{p:an}
Let $G$ be an almost simple group with socle $A_n$ and let $H$ be a maximal subgroup of $G$. Then $d(H) \leqs 4$.
\end{prop}

\begin{proof}
The result can be checked directly if $n=6$, so we may assume $G = A_n$ or $S_n$. First we claim that $d(H) \leqs 3$ in cases (i), (ii), (iii) and (v) of Theorem \ref{t:ons}. If $H$ is almost simple, then $d(H) \leqs 3$ by Theorem \ref{t:dl}. Since $[G:A_n] \leqs 2$, it suffices to show that $d(L)= 2$ for $L=S_k \times S_{n-k}$, ${\rm AGL}_{d}(p)$ or $S_k \wr S_t$. 

\spc

 (i) For $L = S_k \times S_{n-k}$, it is easy to see that 
$L = \la ((1,2),x), (y,(1,2)) \ra$, where $x = (\a,\a+1,\ldots, n-k)$ and $y = (\b,\b+1, \ldots, k)$, with $\a=1$ if $n-k$ is odd, otherwise $\a=2$, and similarly $\b=1$ if $k$ is odd, otherwise 
$\b=2$.

\spc

(ii) If $L={\rm AGL}_{d}(p)$ then $L$ has a unique minimal normal subgroup of order $p^d$ and thus the main theorem of \cite{LucM} implies that $d(L) = \max\{2,d({\rm GL}_{d}(p))\} = 2$.

\spc

 (iii) Suppose $L=S_{k} \wr S_{t}$ and let
$(x_{1}, \ldots, x_{t} ; y)$ denote a general element of $L$, where $x_{i} \in S_{k}$ and $y \in S_{t}$. Set $\a=1$ if $k$ is odd, otherwise $\a=2$. If $t=2$ then $L=\langle x,y \rangle$, where $x=((1,2),(\a, \ldots, k); 1)$ and $y=(1,1;(1,2))$. Similarly, if $t \geqs 4$ is even then it is easy to check that $L = \langle x,y \rangle$ where
$$x=((1,2), 1, \ldots, 1; (2, \ldots, t)),\; y=(1,1,(\a,\ldots, k), 1, \ldots, 1; (1,2)).$$
Similar generators can be given when $t$ is odd. 

\spc

To complete the proof, we may assume $H$ is a diagonal-type subgroup as in part (iv) of Theorem \ref{t:ons}. Let $\Omega$ be the set of cosets of $\{(a, \ldots, a)\,:\, a \in A\}$ in $A^{k}$ and observe that the embedding of $H$ in $G$ arises from the action of $H$ on $\O$.

If $H=A^{k}.({\rm Out}(A) \times S_{k})$ then $A^k$ is the unique minimal normal subgroup of $H$, so \cite{LucM} yields $d(H)=\max\{2,d({\rm Out}(A) \times S_{k})\}$. The structure of the soluble group ${\rm Out}(A)$ is well understood and it is easy to show that 
$d({\rm Out}(A) \times S_{k}) \leqs 4$. 

Finally, suppose $G=A_n$ and $[A^{k}.({\rm Out}(A) \times S_{k}):H]=2$. First assume $k \geqs 3$. One checks that $(1,2) \in S_{k}$ induces an even permutation on $\Omega$, so $H=A^{k}.(J \times S_{k})$ with $[{\rm Out}(A):J]=2$ and we deduce that $d(H) = \max\{2,d(J \times S_{k})\} \leqs 4$. Now assume $k=2$. Here we calculate that $(1,2) \in S_2$  has precisely $\ell=\frac{1}{2}(|A|-i_2(A)-1)$ $2$-cycles on $\Omega$, where $i_2(A)$ is the number of involutions in $A$. Now, if $\ell$ is odd then $H = T^{2}.{\rm Out}(T)$ and thus $d(H) \leqs d({\rm Aut}(T))+1 \leqs 4$. On the other hand, if $\ell$ is even then $H=T^2.(J \times S_{2})$ with $[{\rm Out}(T):J]=2$ and as before we conclude that $d(H)=\max\{2,d(J \times S_{2})\} \leqs 4$.
\end{proof}

\begin{ex}\label{ex:diag}
We can construct maximal diagonal-type subgroups of alternating groups that need $4$ generators, which gives another demonstration of the sharpness of the bound $d(H \cap T) \leqs 4$ in Theorem \ref{t:maxgen}. For example, suppose $A={\rm P\O}_{12}^{+}(p^{2f})$ and $k=2$ for some prime $p \geqs 3$ and positive integer $f$. Then one can show that $H=A^2.({\rm Out}(A) \times S_2)$ is a maximal subgroup of $G=A_{n}$ 
(with $n=|A|$) and 
$$d(H)=\max\{2,d({\rm Out}(A) \times S_2)\} = d(D_8 \times Z_{2f} \times Z_2) = 4.$$
\end{ex}

\subsection{Random generation}

In the previous section we considered the minimal generation of maximal subgroups of 
simple (and almost simple) groups, with the aim of extending Theorem \ref{t:sd2}. In a similar spirit, we now turn to the random generation of these subgroups. 

Let us recall that the main result on the random generation of simple groups is the verification of  Dixon's conjecture (see Conjecture \ref{c:dixon}) and it is natural to ask if an appropriate analogue holds for maximal subgroups of simple groups. It is immediately clear that some modifications are required. Indeed, arbitrarily large maximal subgroups $H$ can have subgroups of bounded index, which prevents $\mathbb{P}_k(H)$ from tending to $1$ as $|H| \to \infty$, for any fixed $k$. For example, $H = S_{n-2}$ is a maximal subgroup of $A_n$ and we have $\mathbb{P}_k(H) \leqs 1 - 2^{-k}$ since a randomly chosen element of $H$ lies in $A_{n-2}$ with probability $1/2$.

The following result (see \cite[Corollary 4]{BLS1}) can be viewed as a best possible analogue of Dixon's conjecture for maximal subgroups of almost simple groups.

\begin{thm}\label{c:eps}
For any given $\e>0$ there exists an absolute constant $k=k(\e)$ such that 
$\mathbb{P}_k(H)>1-\e$ for any maximal subgroup $H$ of an almost simple group.
\end{thm} 

To describe the main steps in the proof, we need some additional notation. Let $G$ be a finite group and set 
$$\nu(G) = \min\{k \in \mathbb{N} \,:\, \mathbb{P}_k(G) \geqs 1/e\}.$$ 
Up to a small multiplicative constant, it is known that $\nu(G)$ is the expected number of random elements generating $G$ (see \cite{Pak} and \cite[Proposition 1.1]{Lub}). If $A$ is a non-abelian chief factor of $G$, let ${\rm rk}_A(G)$ be the maximal number $r$ such that a normal section of $G$ is the direct product of $r$ chief factors of $G$ isomorphic to $A$, and let $\ell(A)$ be the minimal degree of a faithful transitive permutation representation of $A$. Set
\begin{equation}\label{e:delt}
\delta(G) = \max_{A}\left\{\frac{\log {\rm rk}_A(G)}{\log \ell(A)}\right\},
\end{equation}
where $A$ runs through the non-abelian chief factors of $G$. 

We can now state a remarkable theorem of Jaikin-Zapirain and Pyber \cite[Theorem 1.1]{JP}, which is the key ingredient in the proof of Theorem \ref{c:eps}. 

\begin{thm}\label{t:jp}
There exist absolute constants $\a,\b \in \mathbb{N}$ such that 
$$\a(d(G)+\delta(G)) < \nu(G) < \b d(G) + \delta(G)$$
for any finite group $G$. 
\end{thm}

Let $H$ be a maximal subgroup of an almost simple group, so $d(H) \leqs 6$ by Theorem \ref{t:maxgen}. By considering the structure of $H$ (with the aid of the aforementioned  subgroup structure theorems), it is not too difficult to show that $H$ has at most three non-abelian chief factors (see 
\cite[Lemma 8.2]{BLS1}) and thus $\delta(H)<1$. Therefore, Theorem \ref{t:jp} implies that $\nu(H)<6\b+1$. 

To complete the proof of Theorem \ref{c:eps}, set $c = 6\b+1$ and fix $\e>0$. Let $m$ be the smallest positive integer such that $(1-1/e)^m<\e$ and set $k=cm$. Then 
$$1-\mathbb{P}_{k}(H) \leqs (1-\mathbb{P}_{c}(H))^m \leqs (1-1/e)^m<\e$$
and thus $\mathbb{P}_{k}(H)>1-\e$ as required.

\subsection{Subgroup growth}\label{s:growth}

Let $G$ be a finite group. We define the \emph{depth} of a subgroup $H$ of $G$ to be the maximal length of a chain of subgroups from $H$ to $G$ (with proper inclusions). In particular, $H$ is maximal if and only if it has depth $1$. We say that a subgroup of depth $2$ is a \emph{second maximal} subgroup of $G$, and so on. It will be convenient to introduce the following notation:
\begin{align*}
\mathcal{M}_k(G) & = \{H \,:\, \mbox{$H \leqs G$ has depth $k$}\} \\
m_{k,n}(G) &  = |\{H \in \mathcal{M}_k(G) \,:\, [G:H] = n \}|
\end{align*}
For example, $m_{1,n}(G)$ is the number of maximal subgroup of $G$ with index $n$. 

For a fixed value of $k$, it is interesting to consider the growth of $m_{k,n}(G)$ as a function of $n$. For example, if $\mathcal{G}$ is an infinite family of finite groups (of unbounded order) then we can ask if there is an absolute constant $c$ such that $m_{1,n}(G)< n^c$ for all $n$ and all $G \in \mathcal{G}$. If this condition holds, then we say that the groups in $\mathcal{G}$ have \emph{polynomial maximal subgroup growth}. For example, if $p$ is a prime and $G = (Z_p)^d$ then $m_{1,p}(G) = (p^d-1)/(p-1) \sim p^{d-1}$, so elementary abelian $p$-groups do not have this property. 

This growth condition arises naturally in the study of profinite groups. Recall that a profinite group $G$ is \emph{positively finitely generated} (PFG) if 
$\mathbb{P}_{k}(G)>0$ for some positive integer $k$, where $\mathbb{P}_k(G)$ is defined in terms of topological generation if $G$ is infinite. By a celebrated theorem of Mann and Shalev \cite{MS}, $G$ is PFG  if and only if it has polynomial maximal subgroup growth. 

The next result is a combination of \cite[Corollaries 5 and 6]{BLS1}.

\begin{thm}\label{t:sg}
Almost simple groups have polynomial maximal and second maximal subgroup growth. That is, there exists an absolute constant $c$ such that 
$$\max\{m_{1,n}(G), m_{2,n}(G)\} < n^c$$ 
for all almost simple groups $G$ and all $n$.
\end{thm}

To prove this, we need the following result, which combines Theorem \ref{t:jp} with a result of Lubotzky (see \cite[Proposition 1.2]{Lub}).

\begin{thm}\label{t:lub}
There exists an absolute constant $\gamma \in \mathbb{N}$ such that 
$$m_{1,n}(G) < n^{\gamma d(G)+\delta(G)}$$
for all finite groups $G$ and all $n \in \mathbb{N}$.
\end{thm}

Let $G$ be an almost simple group. Then $d(G) \leqs 3$ by Theorem \ref{t:dl} and it is easy to see that $\delta(G) = 0$ (see \eqref{e:delt}), so Theorem \ref{t:lub} yields 
$$m_{1,n}(G) < n^{3\gamma}.$$
As previously noted, $\delta(H) < 1$ for all $H \in \mathcal{M}_{1}(G)$, so 
\begin{align*}
m_{2,n}(G) & \leqs \sum_{a|n}m_{1,a}(G) \max\{m_{1,n/a}(H) \,:\, H \in \mathcal{M}_{1}(G),\, [G:H] = a\} \\
& < \sum_{a|n}a^{3\gamma}(n/a)^{6\gamma+1} < n^{6\gamma+2}
\end{align*}
and thus the bound in Theorem \ref{t:sg} holds with $c = 6\gamma+2$. 

The constant $\gamma$ here can be expressed in terms of the undetermined constant $\b$ in Theorem \ref{t:jp}. It would be desirable to have an effective result with an explicit constant, but it seems rather difficult to extract constants from the proof of Theorem \ref{t:jp} in \cite{JP}.

\subsection{Primitive permutation groups}\label{s:prim}

Let $G \leqs {\rm Sym}(\Omega)$ be a finite primitive permutation group with point stabiliser $H = G_{\a}$. Let us consider the relationship between $d(G)$ and $d(H)$. By primitivity, $H$ is a maximal subgroup of $G$ and thus $d(H) \geqs d(G)-1$. On the other hand, Lemma \ref{l:ns} yields
\begin{equation}\label{e:dh}
d(H) \leqs [G:H](d(G)-1)+1
\end{equation}
and we have observed that equality is possible, even when $H$ is a maximal subgroup of $G$ (see Example \ref{ex:max}). Of course, if $G$ and $H$ are the groups in Example \ref{ex:max} then $G$ does not act faithfully on the cosets of $H$ (indeed, $H$ is a normal subgroup of $G$), so this example does not come from a primitive group. Therefore, we may ask if it is possible to improve the bound in \eqref{e:dh} if we assume $H$ is a \emph{core-free} maximal subgroup. The following result is \cite[Theorem 7]{BLS1}.

\begin{thm}\label{t:prim}
Let $G$ be a finite primitive permutation group with point stabiliser 
$H$. Then 
$$d(G)-1 \leqs d(H) \leqs d(G)+4.$$
\end{thm}

The short proof combines Theorem \ref{t:maxgen} and the O'Nan--Scott theorem (for primitive groups), which describes the structure and action of a primitive permutation group in terms of its socle. Some cases are very easy. For example, if $G$ is an affine group or a twisted wreath product, then $G$ has a regular normal subgroup $N$, hence $G = HN$, $H \cap N = 1$ and $d(H) = d(G/N) \leqs d(G)$. If $G$ is almost simple, then $d(H) \leqs 6$ by Theorem \ref{t:maxgen}, so $d(H) \leqs d(G)+4$. We refer the reader to \cite[Section 10]{BLS1} for the remaining diagonal and product-type cases. 

It would be interesting to know if there are any examples in Theorem \ref{t:prim} with $d(H)=d(G)+4$. Note that one would need to prove Conjecture \ref{c:as} in order to rule out any almost simple examples. We refer the reader to \cite{CLRD} for a recent application of Theorem \ref{t:prim} to the study of the exchange relation for generating sets of arbitrary finite groups.

\subsection{Second maximal subgroups and beyond}\label{s:2max}

Let $G$ be a finite group and recall that $M \leqs G$ is a \emph{second maximal} subgroup of $G$ if it has depth 2, that is, if $M$ is maximal in a maximal subgroup of $G$. These subgroups and their overgroups arise naturally in the study of subgroup lattice theory (see P\'{a}lfy \cite{Pal} and Aschbacher \cite{asch2}, for example). 

In this final section, our goal is to extend some of the results discussed in the previous section on maximal subgroups of almost simple groups to second maximal subgroups. For example, we have seen that almost simple groups and their maximal subgroups are $3$-generated and $6$-generated, respectively, so it is natural to ask whether or not this behaviour extends deeper into the subgroup lattice. This question, plus several related problems, is studied in \cite{BLS2} and we will provide a brief overview of the main results.

First let us fix some notation. Let $G$ be an almost simple group and let $M$ be a second maximal subgroup of $G$, so 
$$M<H<G$$
for some maximal subgroup $H$ of $G$. The following lemma provides an important reduction.

\begin{lemma}\label{l:easy}
If either $H$ or $M$ is almost simple, or if ${\rm core}_H(M) =1$, then $d(M) \leqs 10$.
\end{lemma}

\begin{proof}
If either $H$ or $M$ is almost simple, then $d(M) \leqs 6$ by Theorems \ref{t:dl} and \ref{t:maxgen}. If ${\rm core}_H(M) =1$ then $H$ is a primitive permutation group on the set of cosets $M$ in $H$, so 
$$d(M)\leqs d(H)+4 \leqs 10$$
by combining Theorems \ref{t:maxgen} and \ref{t:prim}.
\end{proof}

If our goal is to investigate the existence of an upper bound of the form $d(M) \leqs c$ for some absolute constant $c$, then by the previous lemma we may assume that $M$ contains a non-trivial normal subgroup of $H$. When viewed in terms of the usual subgroup structure theorems for almost simple groups, this condition on $M$ is rather restrictive. Let us illustrate this with a concrete example.

\begin{ex}
Suppose $M<H<G$, where $G=S_n$ and $H = S_k \wr S_t = N.S_t$ with $N=(S_k)^t$ and $k \geqs 5$ (so $n = kt$ or $k^t$). Note that $A = (A_k)^t$ is the unique minimal normal subgroup of $H$. In particular, if $A \not\leqs M$ then ${\rm core}_H(M)=1$ and thus $d(M) \leqs 10$ by Lemma \ref{l:easy}, so let us assume $A \leqs M$. There are two cases to consider.
\begin{itemize}
\item[(i)] $N \leqs M$: Here the maximality of $M$ in $H$ implies that $M = N.J$ for some maximal subgroup $J<S_t$. Now $J$ has $\ell \leqs 2$ orbits on $\{1, \ldots, t\}$ (indeed, if $J$ is intransitive, then $J = S_a \times S_{t-a}$ for some $a$) and Proposition \ref{p:an} implies that $d(J) \leqs 4$. Therefore,  
$$d(M) \leqs d((S_k)^{\ell})+d(J) \leqs 6.$$

\item[(ii)] $N \not\leqs M$: In this situation, $M = (M \cap N).S_t$ and thus $M/A$ is a maximal subgroup of the wreath product $H/A = S_2 \wr S_t$. One can show that every maximal subgroup of $S_2 \wr S_t$ is $6$-generated (see \cite[Lemma 2.7]{BLS2}) and thus 
$$d(M) \leqs d(A_k)+6 = 8.$$
\end{itemize}
We conclude that $d(M) \leqs 10$.
\end{ex}

The main result on the minimal generation of second maximal subgroups is the following (see \cite[Theorem 1]{BLS2}).

\begin{thm}\label{t:2max}
There is an absolute constant $c$ such that $d(M) \leqs c$ for all 
second maximal subgroups $M$ of almost simple groups $G$ with 
\begin{equation}\label{e:spec}
{\rm soc}(G) \not\in \{{\rm L}_{2}(q), {}^2B_2(q), {}^2G_2(q)\}.
\end{equation}
\end{thm}

\begin{re}
As noted in \cite{BLS2}, we can take $c=12$ for the constant in Theorem \ref{t:2max}, unless ${\rm soc}(G)$ is a simple exceptional group of Lie type and $M$ is maximal in a parabolic subgroup of $G$, in which case the conclusion holds with $c=70$. No doubt these estimates for $c$ could be improved. For example, suppose $G = E_8(q)$ and $M$ is a maximal subgroup of a $D_7$-parabolic subgroup $H = QL$ of $G$, where $Q$ is the unipotent radical and $L$ is a Levi factor of $H$. If $M$ contains $Q$, then $M = QK_0Z$ where $K_0$ is a maximal subgroup of ${\rm Spin}_{14}^{+}(q)$, $Z$ is a central torus of rank $1$ and $Q/Q'$ is a $64$-dimensional spin module for ${\rm Spin}_{14}^{+}(q)$. Since $Q'$ is contained in the Frattini subgroup of $Q$, Theorem \ref{t:maxgen} implies that 
$$d(M) \leqs d(Q/Q') + d(K_0Z) \leqs 64+6 = 70.$$
To improve this estimate, one would have to study the action of each maximal subgroup of ${\rm Spin}_{14}^{+}(q)$ on the spin module $Q/Q'$ (a complete list of the relevant maximal subgroups is not available in the literature).   
\end{re}

Let us look more closely at the groups $G$ excluded in \eqref{e:spec}. First observe that in each case, $G$ has a maximal Borel subgroup $B$. If $M$ is a second maximal subgroup and $M \not\leqs B$, then \cite[Theorem 1]{BLS2} states that $d(M) \leqs c$, where $c$ is the constant in the main statement of Theorem \ref{t:2max}. However, it is not too difficult to see that certain maximal subgroups of $B$ have radically different generation properties. 

\begin{ex}
Suppose $G = {\rm L}_{2}(q)$ with $q=p^f$, so $B = (Z_p)^f{:}Z_{q-1}$. If $q-1$ is a prime, then $M = (Z_p)^f$ is a second maximal subgroup of $G$ with $d(M)=f$. In particular, if we take $q - 1$ to be the largest known Mersenne prime, so $q=2^{74207281}$ at the time of writing, then $M = (Z_2)^{74207281}$ is a second maximal subgroup of ${\rm L}_{2}(q)$ with 
$d(M) = 74207281$.
\end{ex}

This example shows that if there are infinitely many Mersenne primes, which is widely believed to be true, then there is no absolute bound on the number of generators for second maximal subgroups of the groups excluded in \eqref{e:spec}. In fact, the following result shows that this question is equivalent to a formidable open problem in Number Theory (although this is weaker than the existence of infinitely many Mersenne primes, a proof is still far out of reach).

\begin{thm}
There is an absolute constant $c$ such that all second maximal subgroups of almost simple groups are $c$-generated if and only if 
$$\{r \,:\, \mbox{$r$ prime and $(q^r-1)/(q-1)$ is prime for some prime power $q$}\}$$
is finite.
\end{thm}

To finish, let us say a few words on subgroups that lie deeper in the subgroup lattice of an almost simple group $G$. Firstly, by essentially repeating the argument in the proof of Theorem \ref{t:sg}, it is easy to show that 
$$m_{3,n}(G) < n^{70\gamma +2}$$
for all $n$ (where $\gamma$ is the constant in Theorem \ref{t:lub}), assuming the condition on ${\rm soc}(G)$ in \eqref{e:spec} is satisfied. In fact, by carefully studying the excluded groups in \eqref{e:spec}, it is possible to show that \emph{all} almost simple groups have polynomial third maximal subgroup growth (see \cite[Theorem 6]{BLS2}). Does this growth behaviour extend to fourth maximal subgroups and beyond?

\begin{prob}
For all $k \in \mathbb{N}$, is there a constant $c = c(k) \in \mathbb{N}$ such that $m_{k,n}(G)< n^c$ for all $n$ and all almost simple groups $G$?
\end{prob}

Returning to minimal generation, it is not too difficult to construct third maximal subgroups of almost simple groups that need arbitrarily many generators (without needing to establish any formidably difficult results in Number Theory!). 

\begin{ex}
Let $p \geqs 5$ be a prime such that $p\equiv \pm 3 \pmod{8}$. Note that infinitely many primes satisfy this congruence condition, by Dirichlet's theorem. The condition on $p$ implies that $S_4$ is a maximal subgroup of ${\rm PGL}_{2}(p)$ and this allows us to construct a third maximal subgroup of $S_{2(p+1)}$ via the following chain: 
$$S_{2(p+1)}>S_2 \wr S_{p+1} >(S_2)^{p+1}.{\rm PGL}_2(p)>(S_2)^{p+1}.S_4 = H.$$
By applying Lemma \ref{l:ns} we conclude that $d(H) \geqs p/24+1$.
\end{ex}

It would be interesting to seek an appropriate extension of Theorem \ref{t:2max} for third maximal subgroups $H$ (and more generally, depth $k$ subgroups) of almost simple groups.

\end{document}